\documentclass[12pt]{amsart}
\usepackage{color,psfrag,epsfig,url,hhline
}
\usepackage[usenames,dvipsnames]{xcolor}
\numberwithin{table}{section}
\newtheorem{theorem}{Theorem}[section]

\newtheorem{lemma}[theorem]{Lemma}

\theoremstyle{definition}
\newtheorem{definition}[theorem]{Definition}

\newtheorem{cor}[theorem]{Corollary}

\theoremstyle{remark}
\newtheorem{remark}[theorem]{Remark}
\numberwithin{equation}{section}
\newfont{\tap}{tap scaled 650}

\def \[{[ }
\def \]{] }

\begin{document}

\title{Double pants decompositions revisited}


\author{Anna Felikson}
\address{Department of Mathematical Sciences, Durham University, Science Laboratories, South Road, Durham, DH1 3LE, UK}
\email{anna.felikson@durham.ac.uk}

\author{Sergey Natanzon}
\address{
National Research University Higher School of Economics (HSE),
20 Myasnitskaya ulitsa, Moscow 101000, Russia; 
 \phantom{wwwwwwwwwwwwwwwwwwwwwwwwwwwwwwwwwwwwww}
Institute for Theoretical and Experimental Physics (ITEP),
 25 B.Cheremushkinskaya, Moscow 117218, Russia;
\phantom{wwwwwwwwwwwwwwwwwwwwwwwwwwwwwwwwwwwww}
Laboratory of Quantum Topology, Chelyabinsk State University, Brat’ev Kashirinykh street 129, Chelyabinsk 454001, Russia.
}
\email{natanzons@mail.ru}

\thanks{Research of the second author is partially supported  by  RFBR grant 13-01-00755,  by the grant Nsh-5138.2014.1 for support of scintific schools and  by Laboratory of Quantum Topology of Chelyabinsk State University (Russian Federation government grant 14.Z50.31.0020).
 }

\begin{abstract}
Double pants decompositions were introduced in~\cite{FN} together with a flip-twist groupoid acting on these decompositions.
It was shown that flip-twist groupoid acts transitively on a certain topological class of the decompositions,
however, recently Randich discovered a serious mistake in the proof. 
In this note we present a new proof of the result, 
accessible without reading  the initial paper.

\end{abstract}

\maketitle

\noindent
MSC classes: 57M50 

\section{Introduction}
Double pants decompositions are introduced in~\cite{FN} as a union of two pants decompositions of the same surface. These decompositions are subject to certain transformations (called ``flips'' and ``handle twists''  generating a groupoid called ``flip-twist groupoid'', see Section~\ref{defs} for the definitions).  

The main result of~\cite{FN} is  stating that the flip-twist groupoid acts transitively on a certain set of double pants decompositions (called ``admissible double pants decompositions''). In the case of closed surfaces, these admissible pants decompositions can be characterised as ones corresponding to Heegaard splittings of a 3-sphere. In other words, the following theorem was proved in~\cite{FN}:

\bigskip

\noindent
{\bf Main Theorem.}
{\it
Let $S$ be a surface of genus $g$ with $n$ holes, where $2g+n>2$. Then the flip-twist groupoid acts transitively on the set of all admissible double pants decompositions of $S$. 
}

\bigskip

It was shown by Randich in his Master Thesis~\cite{R} that the original argument in~\cite{FN} contains a serious mistake. In this short note we present a new proof of the transitivity theorem, thus confirming that the main result of~\cite{FN} holds true.


As the new proof is short and technically easy, we try to keep this note independent of~\cite{FN}: Section~2 
contains all definitions necessary to formulate and prove the main theorem (Section~\ref{hole} is devoted to the mistake in the old proof and is not necessary to establish the result).

\medskip
\noindent
{\bf Acknowledgements.} We are grateful to Andrew Przeworski and Joseph Randich for careful reading of our paper, spotting the mistake in the original proof and communicating it to us.

\section{Definitions}
\label{defs}


\subsection{Pants decompositions}
Let $S=S_{g,n}$ be an oriented surface of genus $g$ with $n$ holes. 
By a {\it curve} on $S$ we will mean a simple closed essential curve considered up to a homotopy of $S$. Given two curves we always assume that there are no ``unnecessary intersections''  (i.e. the homotopy classes of the curves contain no representatives intersecting in the smaller number of points). We denote by $|a\cap b|$ the number of intersections of the curves $a$ and $b$.

\begin{definition}[Pants decomposition] A {\it pants decomposition} $P$ of $S$ is a collection of non-oriented mutually disjoint curves decomposing $P$ into pairs of pants (i.e., into spheres with 3 holes).

\end{definition} 

There are two important types of transformations acting on pants decompositions: 

\begin{definition}[Flips] Let $P=\{c_1,\dots,c_k \}$ be a pants decomposition,
and suppose that $c_i\in P$ belongs to two different pairs of pants.
A {\it flip} of $P$ (in $c_i$) 
is a substitution of $c_i$ by any curve 
such that $|c_i'\cap c_j|=0$  for $j\ne i$  and  $c_i'\cap c_i=2$
(see Fig.~\ref{defs}.a). 

\end{definition}


\begin{definition}[$\mathcal S$-moves] Let $P=\{c_1,\dots,c_k \}$ be a pants decomposition and let $c_i\in P$ be a curve which belongs to a unique  pair of pants.
An {\it $\mathcal S$-move} in $c_i$ is a  substitution of $c_i$ by any curve $c_i'$ such that
$|c_i'\cap c_j|=0$  for $j\ne i$ and  $c_i'\cap c_i=1$
(see Fig.~\ref{defs}.b).

\end{definition} 

It is shown by Hatcher and Thurston~\cite{HT} that flips and $\mathcal S$-moves act transitively on all pants decompositions of a given surface.

\subsection{Double pants decompositions}
\label{double pants}

\begin{definition}[Double pants decomposition]
A {\it double pants decomposition} $(P^a,P^b)$ is a set of two pants decompositions $P^a$ and $P^b$ of the same surface.

\end{definition}

Clearly, flips act on double pants decompositions (we pick up a curve in $P^a$ or $P^b$ and perform the corresponding flip of an ordinary pants decomposition). 

To model $\mathcal S$-moves, \cite{FN} considers the  transformations called {\it handle twists}. To define them, we will use a notion of a {\it handle curve}:

\begin{definition}[Handle curve]
We will say that a curve $c$ on a surface $S$ is a {\it handle curve} if at least one of the connected components of the surface $S\setminus c$ is a torus with one hole (a ``handle''). 
All other curves will be called {\it non-handle curves}.

\end{definition}

\begin{definition}[Handle twists]
Let $(P^a,P^b)$ be a double pants decomposition and let $c\in P_a\cap P_b$ be a handle curve.
Let $H$ be the handle cut out  by $c$, and let $a_1\in P^a$ and $b_1\in P^b$ be the curves contained in $H$. 
 A {\it handle twist} of $a_1$ along $b_1$ is a Dehn twist  along $b_1$ applied to $a_1$, 
see Fig.~\ref{defs}.c.

\end{definition}

\begin{definition}[$FT$-groupoid]
By a {\it flip-twist groupoid} (or $FT$-groupoid) we mean the groupoid acting on double pants decompositions and generated by all flips and handle twists. 

\end{definition}

\begin{definition}[$FT$-equivalent double pants decompositions]
Two double pants decompositions are called {\it $FT$-equivalent} if there is a sequence of flips and handle twists transforming one of these decompositions to another.

\end{definition}

\begin{figure}[!h]
\begin{center}
\psfrag{a}{\small (a)}
\psfrag{b}{\small (b)}
\psfrag{c}{\small (c)}
\epsfig{file=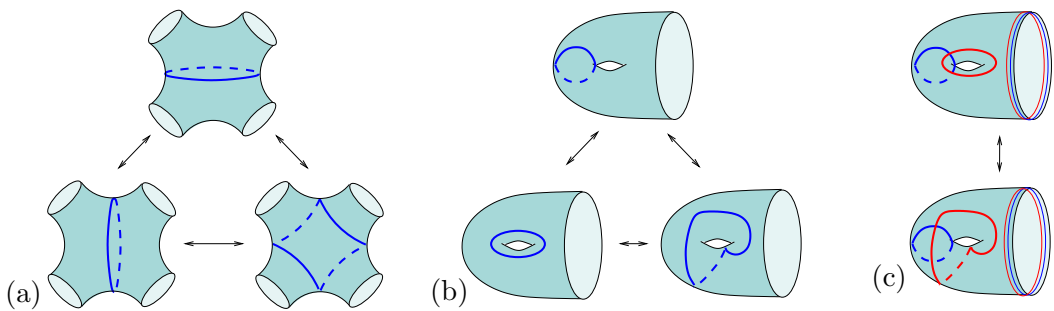,width=0.9\linewidth}
\caption{Examples of (a) flips; (b) $\mathcal S$-moves; (c) handle twist. }
\label{defs}
\end{center}
\end{figure}

\subsection{Admissible double pants decompositions}
\label{admissible}

\begin{definition}[Standard double pants decompositions]
A double pants decomposition $(P^a,P^b)$ of a genus $g$ surface $S=S_{g,n}$ is {\it standard} if there is a set of handle curves $\{c_i\}$ in $P^a\cap P^b$ such that $S'=S\setminus  \{c_i\}$ is a union of $g$ handles $H_1,\dots,H_g$ and at most one sphere with holes, moreover, 
for $a_j,b_j\in H_j$, $a_j\in P^a$, $b_j\in P^b$ we require  $|a_j\cap b_j|=1$.


\end{definition}

\begin{definition}[Admissible double pants decompositions]
A double pants decomposition $(P_1^a,P_1^b)$ is {\it admissible} if it 
may be obtained from a standard double pants decomposition by a sequence of flips.
%

\end{definition}

\begin{remark}[Admissible decompositions as Heegaard splittings of $S^3$]
It is easy to show that in case of closed surfaces admissible double pants decompositions correspond to Heegaard splittings of 3-sphere, see~\cite[Theorem~2.15]{FN} . 
%
%

\end{remark}

The main goal of~\cite{FN} and of the current note is to prove that any two admissible double pants decompositions are $FT$-equivalent.

\section{The issue with  the old proof}
\label{hole}

The proof in~\cite{FN} was based on the notions of {\it zipper system} and {\it zipped flips} (see~\cite[Definitions~1.3, 1.4, 1.7]{FN}). The idea was 1) to show  that every admissible double pants decomposition is compatible with some zipper system; 2)  to prove that all the decompositions compatible with the same zipper system are $FT$-equivalent;
3) to check that for all necessary changes of zipper systems one can use flips and handle twists.
In particular, the first of these steps was based on \cite[Lemma 1.12]{FN} which (wrongly) shows that every flip is a zipped flip.
 
In~\cite{R} Randich shows that  the result of \cite[Lemma 1.12]{FN} is wrong: not every pants decomposition is compatible with a zipper system, and, consequently, not every flip is a zipped flip. To demonstrate this, Randich notices that 
if $P$ is a pants decomposition compatible with a zipper system then the dual graph to $P$ is planar, however, as Randich observes, this property is not always preserved by flips (see Fig.~\ref{ex}).

\begin{figure}[!h]
\begin{center}
\psfrag{c}{ }
\psfrag{c'}{ }
\psfrag{a}{\small (a)}
\psfrag{b}{\small (b)}
\epsfig{file=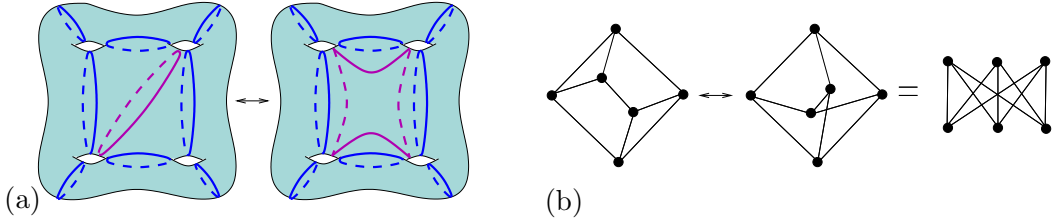,width=0.9\linewidth}
\caption{(a) An example of an unzipped flip; (b) The corresponding transformation of the dual graph results in a non-planar graph }
\label{ex}
\end{center}
\end{figure}

\begin{remark}
In~\cite[Section 3]{FN}, the case of genus 2 was proved without usage of zipper systems, so, is not affected by the detected mistake. This gave the idea for construction of the new proof.

\end{remark}

\section{New proof of the Main Theorem}
\label{proof}

The proof is by induction on the genus of the surface (and on the number of holes for surfaces of the same genus).


\begin{lemma}
\label{genus0}
Flips act transitively on double pants decompositions of $S_{0,n}$. 

\end{lemma}

\begin{proof}
As flips and $\mathcal S$-moves act transitively on (ordinary) pants decompositions~\cite{HT}, 
and since no $\mathcal S$-moves are possible on the sphere, we conclude that flips act transitively on 
(ordinary) pants decompositions of $S_{0,n}$, as well as on double pants decompositions of $S_{0,n}$.

\end{proof}

Lemma~\ref{genus0} settles the base of the induction, genus 0. From now on we consider a surface $S=S_{g,n}$ in assumption that
 $g>0$ and that the main theorem holds for all surfaces $S'=S_{g',n'}$ satisfying $g'<g$ or $g'=g$, $n'<n$.

\begin{lemma}
\label{same handle}
Let $q\subset S$ be a handle curve and suppose that $q\in P^a_1, P^b_1, P^a_2,P^b_2$, where 
$(P^a_1, P^b_1)$ and  $(P^a_2,P^b_2)$ are two admissible double pants decompositions of $S$.
Then $(P^a_1, P^b_1)$ is $FT$-equivalent to  $(P^a_2,P^b_2)$.

\end{lemma}

\begin{proof}
The curve $q$ either cuts $S$ into two smaller surfaces or defines the surface $S\setminus q$ of genus smaller than $g$. In both cases we apply the inductive assumption implying that all admissible double pants decompositions of $S\setminus q$ are $FT$-equivalent.
Performing the  same sequence of transformations on $S$, we see the $FT$-equivalence of all admissible double pants decompositions containing the same handle curve.

\end{proof}

\begin{lemma}
\label{disjoint handles}
Let $q_1$ and $q_2$ be two handle curves in $S$, $q_1 \cap q_2=\emptyset$. 
Let $(P^a_1, P^b_1)$ and  $(P^a_2,P^b_2)$ be admissible double pants decompositions of $S$ such that  $q_1\in P^a_1\cap P^b_1$, $q_2\in  P^a_2\cap P^b_2$. Then  $(P^a_1, P^b_1)$ is $FT$-equivalent to  $(P^a_2,P^b_2)$.

\end{lemma}

\begin{proof}
Cutting $S$ along $q_1$ we obtain a handle and a surface $S'$ of genus $g'<g$, hence by inductive assumption $FT$-groupoid acts transitively on the double pants decompositions of $S'$. In particular,   $(P^a_1, P^b_1)$ is $FT$-equivalent to a standard double pants decomposition  $(P^a_3,P^b_3)$ containing the curve $q_2$ (with $q_2\in P^a_3\cap P^b_3$). In view of Lemma~\ref{same handle} (applied for $q=q_2$) this implies that   $(P^a_1, P^b_1)$ is $FT$-equivalent to $(P^a_2,P^b_2)$.

\end{proof}

Lemma~\ref{same handle} together with Lemma~\ref{disjoint handles} motivate the following definition.

\begin{definition}[$FT$-equivalent handle curves]
Let $q_1$ and $q_2$ be handle curves on $S$. We say that $q_1$ is {\it $FT$-equivalent} to $q_2$ if there exist  double pants decompositions  $(P^a_1, P^b_1)$ and  $(P^a_2,P^b_2)$ such that  $q_1\in P^a_1\cap P^b_1$, $q_2\in  P^a_2\cap P^b_2$, and  $(P^a_1, P^b_1)$ is $FT$-equivalent to $(P^a_2,P^b_2)$.

\end{definition}

In particular, Lemma~\ref{disjoint handles} implies the following corollary.

\begin{cor}
\label{cor disjoint handles}
Any two disjoint handle curves in the same surface are $FT$-equivalent.

\end{cor}

\begin{lemma}
\label{handle for a curve}
If $S$ is a surface of positive genus, then
for every non-handle curve $c\subset S$ there exists a handle curve $q\subset S$ such that $c\cap q=\emptyset$. 

\end{lemma}

\begin{proof}
If $c$ is a separating curve (i.e. $S\setminus c$ is not connected), then at least one of the connected components of  $S\setminus c$ is of positive genus, so, contains a handle curve $q$ disjoint from $c$.

Now, suppose that $c$ is not separating. Then there exists a curve $a$ intersecting $c$ at a unique point. Consider a neighbourhood of $c\cup a$:  its boundary is a simple closed curve (denote it $q$, see Fig.~\ref{nbhood}.a). Moreover, it is easy to check that $q$ is a handle curve, which is clearly disjoint from $c$.

\end{proof}

\medskip
\noindent
{\bf Plan of proof of the theorem:}

\smallskip

\noindent
{\bf 1. (Reduce to standard).} As admissible double pants decompositions are the ones flip-equivalent to the standard ones, it is sufficient to show that any two standard double pants decompositions of the same surface are $FT$-equivalent.

\medskip
\noindent
{\bf 2. (Reduce to one handle).} In view of Lemma~\ref{same handle}, any two standard double pants decompositions containing the same handle curve are $FT$-equivalent. So, to prove that any two standard double pants decompositions are $FT$-equivalent, it is sufficient to prove that any two handle curves $c_{start}$ and $c_{end}$ are $FT$-equivalent. 

\medskip
\noindent
{\bf 3. (On the curve complex, take a path from  $c_{start}$ to $c_{end}$).} Consider the curve complex $\mathcal C(S)$ of $S$ (i.e. the complex whose vertices correspond to homotopy classes of  simple closed curves on $S$ and whose simplices are spanned by vertices corresponding to disjoint sets of curves; in particular, the edges correspond to disjoint pairs of curves).
In view of~\cite{HT} $\mathcal C(S)$ is connected, so,
there exists a sequence $\sigma$ of curves $\{c_{start}=c_0, c_1, c_2, \dots, c_m=c_{end}\}$ such that 
$c_i\cap c_{i+1}=\emptyset$ for $i=1,\dots,m-1$.

\medskip
\noindent
{\bf 4. (Decompose the path into handle-free subpathes).}
Decompose the sequence $\sigma$ into finitely many subsequences $\sigma_1,\dots,\sigma_t$ such that the endpoints of each subsequence are handle curves and all other curves in $\sigma$ are non-handle curves.
%
It is sufficient to prove that the endpoints of one subsequence are $FT$-equivalent:
Corollary~\ref{cor disjoint handles} takes care of transferring from one subsequence to the adjacent one.

%

\medskip
\noindent
{\bf 5. (Choose a disjoint handle for each curve in the subpath). }
Given a subsequence $\sigma_i=\{c_{i,1},\dots ,c_{i,m_i}\}$ (where $c_{i,1}$ and $c_{i,m_i}$ are handle curves while all other curves in $\sigma_i$ are not), for each non-handle curve $c_{i,j}$ ($1<j<m_i$) consider a handle curve $q_{i,j}$ which does not intersect $c_{i,j}$ (it does exist in view of Lemma~\ref{handle for a curve}).
We get a caterpillar as in Fig~\ref{caterpillar} (sitting inside the curve complex).

\medskip
\noindent
{\bf 6. (Move from one leg of the caterpillar to the adjacent one).}
It is left to prove that the handle curve $q_{i,j}$ is $FT$-equivalent to the handle curve $q_{i,j+1}$ (for any $1 \le j<m_i$). This is done in Lemmas~\ref{two handles to the same curve} and~\ref{steps}.

\begin{figure}[!h]
\begin{center}
\psfrag{c1}{$c_1$}
\psfrag{c2}{$c_2$}
\psfrag{q1}{$q_1$}
\psfrag{q2}{$q_2$}
\psfrag{c_2}{$c_2$}
\psfrag{c1}{$c_1$}
\psfrag{c}{$c$}
\psfrag{N}{$N$}
\psfrag{q}{$q$}
\psfrag{s}{$\sigma$}
\psfrag{s1}{$\sigma_1$}
\psfrag{s2}{$\sigma_2$}
\psfrag{s3}{$\sigma_3$}
\psfrag{s4}{$\sigma_4$}
\psfrag{a}{\small (a) }
\psfrag{b}{\small (b) Cor.~\ref{cor disjoint handles}}
\psfrag{c_}{\small (c) L.~\ref{handle for a curve}}
\psfrag{d}{\small (d) L.~\ref{two handles to the same curve}}
\psfrag{e}{\small (e) L.~\ref{steps}}
\epsfig{file=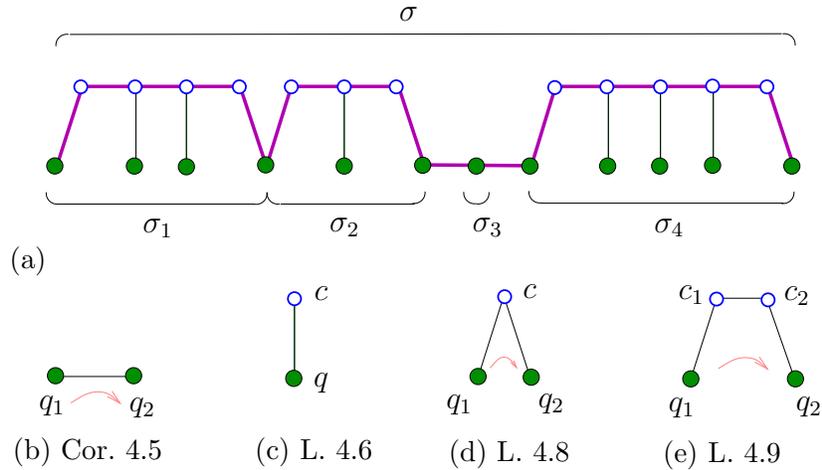,width=0.7\linewidth}
\caption{ (a) Idea of proof: caterpillar; (b)-(e) List of lemmas. \qquad
Filled/empty nodes denote handle/non-handle curves respectively. }
\label{caterpillar}
\end{center}
\end{figure}

\begin{remark}
The idea of a caterpillar-type proof is inspired by~\cite{FZ}, where a caterpillar was used to prove Laurent phenomenon in cluster algebras. 
\end{remark}

\begin{lemma}
\label{two handles to the same curve}
Let $c\subset S$ be a non-handle curve, let $q_1,q_2\subset S$ be two handle curves satisfying $c\cap q_1=c\cap q_2=\emptyset$. Then $q_1$ is $FT$-equivalent to $q_2$.

\end{lemma}

\begin{proof}
Consider $S\setminus c$. If $q_1$ and $q_2$ belong to the same connected component of  $S\setminus c$
then we use an inductive assumption (as all connected components are either of smaller genus or, in assumption of the same genus, have smaller number of holes). If  $q_1$ and $q_2$ belong to different connected components, then $q_1\cap q_2=\emptyset$ and we can use Corollary~\ref{cor disjoint handles}.

\end{proof}

\begin{lemma}
\label{steps}
Let $c_1,c_2\subset S$ be two non-handle curves, $c_1\cap c_2=\emptyset$. Let $q_1,q_2\subset S$ be handle curves such that $c_1\cap q_1=c_2\cap q_2=\emptyset$. Then  $q_1$ is $FT$-equivalent to $q_2$.

\end{lemma}

\begin{proof}
If $g>2$ (where $g$ is the genus of $S$), then at least one connected component of $S\setminus \{ c_1,c_2\}$ is a surface of positive genus, so, there exists a handle curve $q\in S \setminus \{ c_1,c_2\}$ which does not intersect $c_1\cup c_2$. By Lemma~\ref{two handles to the same curve}, $q_1$ is $FT$-equivalent to $q$, and $q$ is $FT$-equivalent to $q_2$, so the statement follows (see Fig.~\ref{nbhood}.b).

To prove the lemma for $g=1,2$,  we will consider three cases: 
either both $c_1$ and $c_2$ are separating, or just one of them, or neither.

\noindent
{\bf Case 1: both $c_1$ and $c_2$ are separating.} Then  $S\setminus \{c_1,c_2\}$ has a connected component of a positive genus, and, as above, there is a handle curve $q$ in that component, such that  $q\cap \{c_1\cup c_2\}=\emptyset$. Thus, the statement follows again from  Lemma~\ref{two handles to the same curve}  (see Fig.~\ref{nbhood}.b).

\noindent
{\bf Case 2: $c_1$ is separating, $c_2$ is not separating.} Consider $S'=S\setminus c_1$.
Notice that the connected component of $S'$ containing $c_2$ has a positive genus. 
So, by Lemma~\ref{handle for a curve} there exists a handle curve $q\subset S'$ disjoint from $c_2$.
Since $q$ is also disjoint from $c_1$, we may apply  Lemma~\ref{two handles to the same curve} as in Fig.~\ref{nbhood}.b again.


\begin{figure}[!h]
\begin{center}
\psfrag{a1}{\color{blue} $a$}
\psfrag{c2}{$c_2$}
\psfrag{c_2}{\color{red} $c$}
\psfrag{c1}{$c_1$}
\psfrag{c}{$c$}
\psfrag{N}{$N$}
\psfrag{q1}{$q_1$}
\psfrag{q}{$q$}
\psfrag{q12}{$q$}
\psfrag{q12_z}{\color{OliveGreen} $q$}
\psfrag{q2}{$q_2$}
\psfrag{1}{$(P^a_1, P^b_1)$}
\psfrag{2}{$(P^a_2, P^b_2)$}
\psfrag{a}{\small (a)}
\psfrag{b}{\small (b)}
\psfrag{c_}{\small (c)}
\psfrag{d}{\small (d)}
\epsfig{file=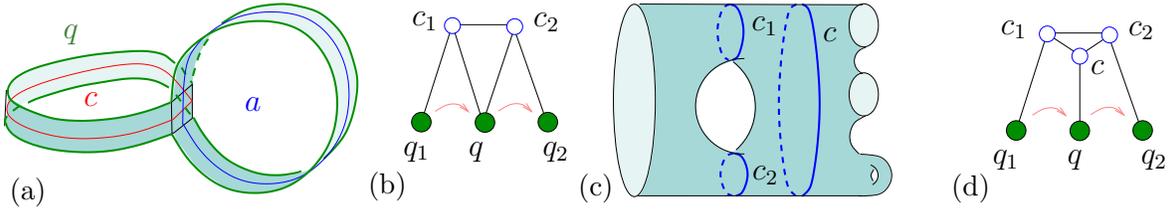,width=0.99\linewidth}
\caption{To the proof of Lemmas~\ref{handle for a curve} and~\ref{steps}.}
\label{nbhood}
\end{center}
\end{figure}

\noindent
{\bf Case 3: neither $c_1$ nor $c_2$ are separating.} 
We consider two possibilities: either  $S'= S\setminus \{c_1,c_2\}$ is not a disjoint union  $S_{0,3}\sqcup S_{0,3}$ of two pairs of pants
or it is.



\noindent
{\bf 3.1.}
Suppose that  $S'= S\setminus \{c_1,c_2\}$ is not a disjoint union of two pairs of pants (i.e. the surface is bigger than the one on Fig.~\ref{torus}.a). Then
$S'$ contains a separating  curve $c$. If $c$ is a handle curve, then we are in the settings of Fig.~\ref{nbhood}.b again (with $q=c$). If $c$ is not a handle curve, then we can insert $c$ into the sequence $\sigma$ between $c_1$ and $c_2$ (as in  Fig.~\ref{nbhood}d), use  Lemma~\ref{handle for a curve}  to construct a handle curve $q$ disjoint from $c$, and finally use Case 2 of the proof to show that $q_1$ is $FT$-equivalent to $q$ and $q$ is $FT$-equivalent to $q_2$.

\noindent
{\bf 3.2.}
Now, suppose that  $S'= S\setminus \{c_1,c_2\}$ is a union of two disjoint pairs of pants, as in Fig.~\ref{torus}.a. Let $q_1$ and $q_2$ be the handle curves shown in  Fig.~\ref{torus}.b and~\ref{torus}.c; notice  that $q_1$ and $q_2$ are disjoint from $c_1$ and $c_2$ respectively. We are left to prove that $q_1$ is $FT$-equivalent to $q_2$,
i.e. that there are  double pants decompositions $(P^a_1, P^b_1)$  and  $(P^a_2,P^b_2)$ such that  $c_1\in P^a_1\cap P^b_1$, $c_2\in  P^a_2\cap P^b_2$, and  $(P^a_1, P^b_1)$ is $FT$-equivalent to $(P^a_2,P^b_2)$.
An example of these double pants decompositions together with a sequence of $FT$-transformations is shown in   Fig.~\ref{torus}.d.

\end{proof}

\begin{figure}[!h]
\begin{center}
\psfrag{c1}{$c_1$}
\psfrag{c2}{$c_2$}
\psfrag{q1}{$q_1$}
\psfrag{q2}{$q_2$}
\psfrag{flip}{\color{blue} \it flip}
\psfrag{flip1}{\color{red} \it flip}
\psfrag{1}{$P^a_1=$}
\psfrag{2}{$P^b_1=$}
\psfrag{3}{$=P^a_2$}
\psfrag{4}{$=P^b_2$}
\psfrag{a}{\small (a)}
\psfrag{b}{\small (b)}
\psfrag{c}{\small (c)}
\psfrag{d}{\small (d)}
\epsfig{file=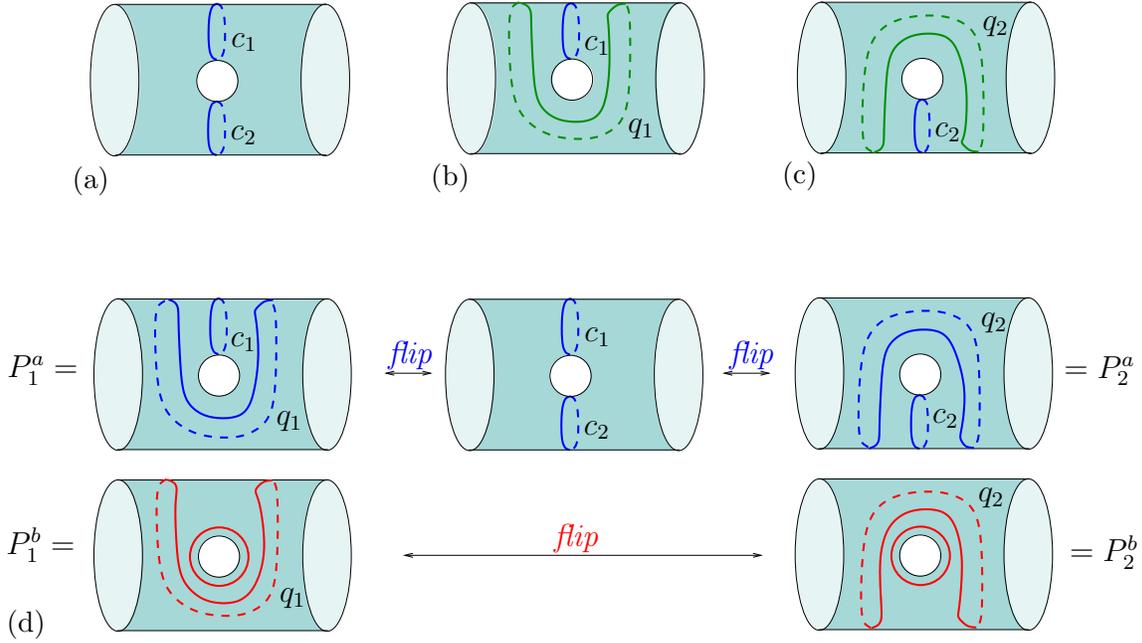,width=0.99\linewidth}
\caption{To the proof of Case~3.2. }
\label{torus}
\end{center}
\end{figure}

\end{document}